\documentclass[a4paper]{article}

\usepackage[margin=3cm,footskip=1cm]{geometry}

\usepackage{amsmath,amssymb,amsthm,bm,mathtools,xspace,booktabs,url}
\usepackage{graphicx,etoolbox}

\usepackage[shortlabels]{enumitem}

\setlist{
  listparindent=\parindent,
  parsep=0pt,
}

\usepackage[sort&compress]{natbib}

\AtBeginEnvironment{thebibliography}{
  \small
  \setlength\itemsep{0.75em plus 0.5em minus 0.75em}%
}

\usepackage{caption}
\usepackage[raggedright]{titlesec}

\numberwithin{equation}{section}

\theoremstyle{plain} 
\newtheorem{theorem}{Theorem}[section]
\newtheorem{Lemma}[theorem]{Lemma}
\newtheorem{Proposition}[theorem]{Proposition}

\theoremstyle{definition} 
\newtheorem{Example}[theorem]{Example}
\newtheorem{Remark}[theorem]{Remark}

\usepackage{authblk}

\makeatletter
\newcommand\CorrespondingAuthor[1]{
  \begingroup
  \def\@makefnmark{}
  \footnotetext{Corresponding author: #1}
  \endgroup
}
\makeatother

\makeatletter
\renewenvironment{abstract}{%
  \small%
  \providecommand\keywords{%
    \par\medskip\noindent\textit{Keywords:}\xspace}%
  \begin{center}%
    \bfseries \abstractname\vspace{-.5em}\vspace{\z@}%
  \end{center}%
  \quote%
}
{
\endquote}
\makeatother

\usepackage[above,below,section]{placeins}

\usepackage[all]{onlyamsmath}

\usepackage{lipsum}

\usepackage{xcolor}

\usepackage{mathrsfs}
\DeclareMathAlphabet{\mathpzc}{T1}{pzc}{m}{it}

\newcommand{\RL}{\mathbb{R}}
\newcommand{\nat}{{\mathbb N}}

\newcommand{\e}{\mathrm{e}}
\newcommand{\dd}{\mathrm{d}}

\newcommand{\oh}{{\mathit{o}}}
\newcommand{\Oh}{{\mathit{O}}}

\newcommand{\Prob}{\mathbb{P}}

\newcommand{\Ftail}{\overline F\,}
\newcommand{\Gtail}{\overline G\,}
\newcommand{\Htail}{\overline H}

\def\*{\discretionary{}{\hbox{\ensuremath\cdot}\thinspace}{}}

\begin{document}

\title{Asymptotic tail behavior of phase-type scale mixture distributions}
\author{Leonardo Rojas-Nandayapa}
\author{Wangyue Xie}
\affil{\footnotesize School of Mathematics and Physics, The University of Queensland, 
   Brisbane, QLD, Australia, \authorcr l.rojas@uq.edu.au, w.xie1@uq.edu.au}
%\author{Leonardo Rojas-Nandayapa\\School of Mathematics and Physics\\The University of Queensland\\
%   Brisbane, QLD, Australia\\l.rojas@uq.edu.au
%  \and Wangyue Xie \\School of Mathematics and Physics\\The University of Queensland\\
%   Brisbane, QLD, Australia\\w.xie1@uq.edu.au}
\date{}
\maketitle
\begin{abstract}

We consider \emph{phase-type scale mixture} distributions 
which correspond to distributions of a product of two independent random variables:  
a phase-type random variable $Y$ and a nonnegative but  
otherwise arbitrary random variable $S$ called 
the \emph{scaling random variable}.  
We investigate conditions for such a class of distributions
to be either light- or heavy-tailed, we explore subexponentiality
and determine
their maximum domains of attraction.
Particular focus is given to  phase-type scale mixture distributions where 
the scaling random variable $S$ has discrete support ---
such a class of distributions has been recently used in risk applications
to approximate heavy-tailed distributions.
Our results are complemented with several examples.


\end{abstract}

\section{Introduction}
In this paper, we consider the class of nonnegative distributions defined by
the \emph{Mellin--Stieltjes convolution} \citep{Bingham1987} 
of two nonnegative distributions $G$ and $H$, given by
\begin{equation}\label{PHSM}
 F(x)=\int_0^\infty G(x/s) \dd H(s),\qquad x\ge0.
\end{equation} 
A distribution of the form \eqref{PHSM} will be called a 
\emph{phase-type scale mixture} if $G$ is a (classical) phase-type (PH) distribution \citep[cf.][]{Latouche1999}
and $H$ is a proper nonnegative distribution that we shall call the 
\emph{scaling distribution}.
A phase-type scale mixture distribution can be seen as the distribution
of a random variable $X:=S\cdot Y$
where $S\sim H$ and $Y\sim G$; accordingly, $S$ is referred as the 
\emph{scaling random variable}.  This terminology is also explained using conditional arguments: 
observe that $(X|S=s) \sim G_s$ where $G_s(x):=G(x/s)$ corresponds to the distribution of the 
(scaled) random variable $s\cdot Y$ which is itself a 
PH distribution, so the distribution $F$ can be thought as a mixture of 
the scaled PH distributions in $\{G_s:s>0\}$ with respect to 
the scaling distribution $H$.

Our motivation for studying the tail behavior of phase-type scale 
mixtures is their use for approximating heavy-tailed distributions in risk applications \citep{Bladt2014b}.
To introduce such an approach, we shall first recall that the family of (classical) phase-type (PH) 
distributions, which corresponds to 
distributions of absorption times of Markov jump processes with one absorbing state 
and a finite number of transient states.
%(the exponential, Erlang and hyperexponential distributions are examples of PH distributions). 
The PH class is particularly attractive since it is tractable and possesses many desirable properties 
(densities, cumulative distributions, moments and integral transforms 
have closed-form expressions in terms of matrix exponentials; it is a closed class under 
scaling, finite mixtures and finite convolutions (cf.~\cite{Assaf1982,O'Cinneide1992}). 
The PH class is popular for modelling purposes because it is dense in the nonnegative 
distributions \citep[cf.][]{Asmussen2003}, so one could in principle approximate 
any nonnegative  distribution with an arbitrary precision.  
This classical approach has been widely studied and  
reliable methodologies for approximating nonnegative
distributions are already available 
\citep[cf.][]{AsmussenNermanOlsson96}.

However, distributions in the PH class are light-tailed
and belong to the Gumbel domain of attraction exclusively \citep{Kang1999}.
Therefore, the PH class cannot 
correctly capture the characteristic behavior of a heavy-tailed distribution 
in spite of its denseness.  In fact,  
this approach may deliver unreliable approximations for important quantities of interest,
such as the ruin probability of a Cram\'er--Lundberg risk process with heavy-tailed claim size
distributions
\citep{VatamidouAdanVlasiouZwart2014}.
As an alternative, the PH class has been extended to distributions
of absorption times having a countable number of transient states
\citep[this approach is attributed to][]{Neuts1981}. 
The later class, which goes under the name of infinite dimensional phase-type distributions
(IDPH), is known to contain heavy-tailed distributions.  
Nevertheless, the IDPH class is no longer mathematically tractable and it is not 
fully documented yet (to the best of the
authors' knowledge, one of the few published references available outlining its mathematical properties
is \cite{Dinghua1996}; another reference of interest is \cite{Greiner1999}, who 
consider infinite mixtures of exponential distributions 
to approximate power-tailed distributions).

To address this issue, \cite{Bladt2015} propose the use of phase-type scale 
mixtures having discrete scaling distributions to approximate heavy-tailed distributions. 
Such a class forms a structured subfamily of the IDPH class that contains the PH class,
so it is trivially dense in the nonnegative distributions.  Two important advantages over the more general
IDPH class are that the class of phase-type scale mixture distributions
is mathematically tractable and that it contains a rich variety of heavy-tailed distributions.

The class of phase-type scale mixture distributions has great potential in applications in engineering,
finance and specifically in insurance.  
As an example of the later, 
\cite{Bladt2015} provide  renewal results
that can be applied to obtain exact expressions for the ruin probability
of a classical Cram\'er--Lundberg risk process having claim sizes distributed according
to a phase-type scale mixture distribution with discrete scaling. 
This approach is further explored in 
\cite{Nardo2016}, where a systematic methodology for approximating arbitrary heavy-tailed
distributions via phase-type scale mixtures is provided; such a formulation provides
simplified formulas for approximating ruin probabilities with arbitrary claim size distributions.  
Furthermore, \cite{BladtRojas2017} provide statistical inference procedures 
based on the EM algorithm to adjust phase-type scale mixtures to heavy-tailed data/distributions. 
Other references of interest that apply similar ideas to risk models include \cite{HashorvaPakesTang2010} and \cite{VatamidouAdanVlasiouZwart2012}.  
%They gave the asymptotic tail probability 
%of the random scaling given that the scaling distribution is in the Gumbel domain of 
%attaction.  They also gave a condition where the claim size distributions are both Gumbel type 
%and subexponential
%and provide the asymptotic probability of ruin for a discrete-time risk model.

In spite of the denseness  and 
the mathematically tractability of the class of phase-type 
scale mixtures, the tail properties of the proposed class 
are not fully understood yet; this paper  concentrates on this issue.
In particular, a key aspect in the successful  approximation of heavy-tailed distributions
via phase-type scale mixtures is the appropriate selection of the scaling distribution.
This paper focuses on the theoretical foundations justifying the selections made in some 
of the applications mentioned above, as well as on providing general guidelines for selecting
appropriate scaling distributions.
We collect and adapt some known results which are available in different contexts, 
and we prove new results that
will allow us to provide a characterization of the tail behavior of phase-type scale mixtures,
as well as a classification of their maximum domains of attraction. 
We expect our results to be useful for modelling purposes by providing a better understanding of the advantages
and limitations of such an approach,
as well as providing criteria for
selecting appropriate scaling distributions for approximating general heavy-tailed distributions.  
Our results are summarized below.

Firstly, we concentrate on classifying light- and heavy-tailed distributions.
A phase-type scale mixture
is heavy-tailed if and only if its scaling distribution has unbounded support.
An interesting heuristic interpretation of this result is as follows: a PH random variable  
multiplied with a
random variable $S$ is heavy-tailed iff $S$ has unbounded support. 
We provide a simple proof of this fact
but we remark that a proof (unknown to us until recently) was already provided in a different context \citep[cf.][]{SuChen2006,Tang2008b}.

Secondly, we focus on the maximum domains of attraction and subexponential properties
of the class of phase-type scale mixtures.
A classical result for the Fr\'echet case 
is Breiman's lemma \citep{Breiman1965}, which  implies that  
a phase-type scale mixture with a regularly varying scaling distribution
remains regularly varying with the same index (hence subexponential).
An analogue closure property exists for the class of Weibullian
distributions \citep{ArendarczykDebicki2011}.
In addition, we investigate analogue results for  
scaling distributions in the Gumbel domain of attraction. We  
show that if a certain higher order derivative of the Laplace--Stieltjes
transform of the reciprocal of the scaling random variable
$\mathcal{L}_{1/S}(\theta)$
is a von Mises function, then $F\in\mathrm{MDA}(\Lambda)$; in addition, we provide
a verifiable condition for subexponentiality.

We then specialize in phase-type scale mixture distributions having discrete support.
Such a class of distributions is of critical importance in applications due to its mathematical tractability, 
as these correspond to distributions
of the absorption time of a Markov jump process having an infinite number of transient states.
We outline a simple methodology which allows us to determine their asymptotic behavior by constructing
a phase-type scale mixture distribution with continuous scaling and having an
asymptotically proportional tail probability.
This methodology can be \emph{reverse-engineered} so we can construct discrete scaling distributions
for approximating the tail probability of some arbitrary target distributions.

The rest of the paper is organized as follows. In Section \ref{mysec2}, 
we set up notations and  summarize some of the standard facts
on heavy-tailed,  phase-type and related  distributions.  
Then we introduce the class of phase-type scale mixtures and examine some of its asymptotic properties. 
Our main results are presented in Section \ref{mysec3} and \ref{mysec4}.  Section \ref{mysec3} is devoted to the
general case, while Section \ref{mysec4} is specialized in discrete scaling distributions.
In Section \ref{mysec5}, we present our conclusions.

\section{Preliminaries}
\label{mysec2}

In this section we provide a summary of some of the concepts needed for this paper.
Most results in this section are  standard. 
A reader familiar with phase-type distributions and 
extreme value theory can safely skip to subsection \ref{PHM}.

%In subsection \ref{sub.PHdistributions} 
 % In subsection \ref{sub.HTdistributions}

%\subsection{Phase-type scale mixtures and related distributions}
%\label{sub.PHdistributions}
First we consider the class of \emph{phase-type} 
(PH) distributions.  When a distinction is needed, we will 
refer to this class of distributions as \emph{classical}, in order to make
a clear distinction from the class of phase-type scale mixture distributions.
%we introduce classical phase-type distributions.
A classical phase-type distribution corresponds to the distribution of 
the absorption time  of a Markov jump process 
$\{X_t\}_{t\geq0}$ with a finite transient state space $E=\{1,2,\cdots,p\}$
and one absorbing state $0$
%Let $Y:=\inf\{t\geq 0:X_t=0\}$ be the random variable of the time until 
%absorption in state $0$. The distribution of $Y$ is called \emph{phase-type 
%distribution} 
\citep[cf.][]{Latouche1999,Asmussen2003}.
Phase-type distributions are characterized by a $p$-dimensional row vector $\boldsymbol{\beta}=
(\beta_1,\cdots,\beta_p)$ (corresponding to the probabilities of starting
the Markov jump process in each of the transient states), and
an intensity matrix 
\begin{equation*}
\mathbf{Q}=\left(\begin{array}{cc}0 &\mathbf{0}  \\
             \boldsymbol{\lambda}  & \mathbf{\Lambda} \end{array}\right),
\end{equation*}
where $\bm{\Lambda}$ is a $p\times p$ sub-intensity matrix. Since rows 
in a intensity matrix must sum to $0$, we also have $\bm{\lambda}=-\bm{\Lambda e}$, 
where $\bm{e}$ is the $p-$dimensional column vector of $1$s.
Phase-type distributions are denoted $\mbox{PH}(\boldsymbol{\beta},\boldsymbol{\Lambda})$, 
and their cumulative distribution functions  are given by
\begin{equation*}
 %g(x)=\boldsymbol{\beta}\e^{\mathbf{\Lambda}x}{\boldsymbol{\lambda}},\qquad
 G(x)=1-\boldsymbol{\beta}\e^{\mathbf{\Lambda}x}\mathbf{e},\qquad \forall x>0.
\end{equation*}

In this paper, we are particularly interested in distributions of scaled phase-type random variables
$s\cdot Y$ where  $Y\sim\mathrm{PH}(\boldsymbol{\beta},\mathbf{\Lambda})$ and $s>0$. 
From the expression above, it follows easily 
that $s\cdot Y\sim\mathrm{PH}(\boldsymbol{\beta},\mathbf{\Lambda}/s)$, so the class
of phase-type distributions is closed under scaling transformations. 
The following is a well known result describing the tail behavior of phase-type distributions
\citep[cf.][]{Asmussen2003}:
\begin{Proposition}
\label{myprop2.1}
Let $G_s\sim\mathrm{PH}(\boldsymbol{\beta}$, $\mathbf{\Lambda}/s$). 
The tail probability of $G_s$ can be written as
\begin{equation*}
\Gtail_s(x)=\sum_{j=1}^{m}\sum_{k=0}^{\eta_j-1}\left(\frac{x}{s}\right)^k\e^{\Re(-\lambda_j) x/s}
 \bigg[c_{jk}^{(1)}\sin(\Im(-\lambda_j)x/s)+c_{jk}^{(2)}\cos(\Im(-\lambda_j) x/s)\bigg].
\end{equation*}
Here $m$ is the number of Jordan blocks of the matrix $\mathbf{\Lambda}$,
$\{-\lambda_j:j=1,\dots,m\}$ are the corresponding eigenvalues and 
$\{\eta_{j}:j=1,\dots,m\}$ the dimensions of the Jordan blocks.
The values $c_{jk}^{(1)}$, $c_{jk}^{(2)}$ are constants depending on the initial distribution $\boldsymbol{\beta}$,
the dimension of the $j$-th Jordan block $\eta_j$ and the generalized eigenvectors of $\mathbf{\Lambda}$.
\end{Proposition}
All eigenvalues of a sub-intensity matrix $\mathbf{\Lambda}$ have
negative real parts and the one with the largest absolute value is always real.  Therefore,
the asymptotic behavior of a scaled phase-type distribution 
is determined by the largest eigenvalue and the largest dimension among
the Jordan blocks associated to the largest eigenvalue 
\citep[see also][]{Asmussen2003,Alexandru2006}. 
%,  that
%is if $G_s\sim\mbox{PH}(\beta,\Lambda/s)$, then 
%%; in such a case the following result can be easily
%%deduced from the previous Proposition:
%%\begin{Corollary}\label{AsympPH}
%%Let $G_s\sim\mathrm{PH}(\boldsymbol{\beta},\Lambda/s)$. If $\Lambda$ has real eigenvalues, then 
%\begin{align*}
%\Gtail_s(x)&=\frac{\gamma x^{\eta-1}\e^{-\lambda x/s}}{s^{\eta-1}}\left(1+o(1)\right),\qquad x\to\infty,
%%\\g_s(x)&=\frac{\mu x^{\eta-1}\e^{-\lambda x/s}}{s^\eta}\left(1+o(1)\right),\qquad  x\to\infty,
%\end{align*}
%where $-\lambda$ is the largest real eigenvalue of $\mathbf{\Lambda}$, $\eta$ is the largest dimension of 
%the Jordan block associated to $-\lambda$
%and $\gamma$ is a positive constant.
It is also well known that if the sub-intensity matrix  $\mathbf\Lambda$ is irreducible, then the tail probabilities of phase-type distributions
decay exponentially \citep[cf.\, Proposition IX.1.8][]{AsmussenAlbrecher2011}.
%$\Gtail_s(x)\sim C\e^{-\lambda x}$ for some constant $C$ 
%and $\lambda$ begin the largest eigenvalue.  
The assumption that the matrix $\mathbf{\Lambda}$ is not irreducible can be 
further relaxed if all eigenvalues are real.
%\end{Corollary}
Also notice that if all the eigenvalues of $\mathbf{\Lambda}$ 
are real ($\Im(-\lambda_j)=0$), then 
\begin{equation}
\label{PH tail}
\Gtail_s(x)=\sum_{j=1}^{m}\sum_{k=0}^{\eta_j-1}c_{jk}\left(\frac{x}{s}\right)^k\e^{-\lambda_j x/s}.
%  \bigg[c_{jk}^{(1)}\sin(\Im(\lambda_j)x/s)+c_{jk}^{(2)}\cos(\Im(\lambda_j) x/s)\bigg].
\end{equation}

Next, we introduce the class of heavy-tailed distributions that will be used
in this paper (various other definitions of heavy-tailed distributions are available
in the literature) and discuss several important subfamilies of heavy-tailed distributions.
We also provide a brief summary of results connecting extreme value theory
with heavy-tailed distributions and subexponentiality.

We say that a nonnegative distribution $H$ is \emph{heavy-tailed}  if
%\begin{equation*}
 %\int\e^{-\theta s}\dd H(s)=\infty,\qquad \forall \theta<0.
%\end{equation*}
\begin{equation*}
\limsup\limits_{s\to\infty}\Htail(s)\e^{\theta s}=\infty,\quad\forall\theta>0,
\end{equation*}
where $\Htail(s)=1-H(s)$ is the \emph{tail probability} of the distribution $H$.
Otherwise, we say that $H$ is a \emph{light-tailed} distribution.
The definition of light/heavy-tailed distributions is often considered 
too general for most practical purposes and
it is more common to work instead with certain families of 
distributions.  For instance,
the so-called \emph{Embrechts--Goldie} class of distributions \citep{EmbrechtsGoldie1980}, 
denoted $\mathcal{L}(\lambda)$,
consists of nonnegative distributions $H$ having the property
\begin{equation*}
\lim_{s\to\infty}\dfrac{\Htail(s-t)}{\Htail(s)}=\e^{\lambda t}, \quad \lambda\geq 0,\forall{t}.
\end{equation*}
Distributions in the class $\mathcal{L}(0)$ are
heavy-tailed  and these are known as long-tailed distributions. 
In contrast, if $\lambda>0$ then a distribution in the class $\mathcal{L}(\lambda)$ is light-tailed. 
From Proposition \ref{myprop2.1}, it is clear that a PH distribution is in $\mathcal{L}(\lambda)$ where $-\lambda$
is the largest eigenvalue of the sub-intensity matrix $\boldsymbol{\Lambda}$.

An important subclass of heavy-tailed distributions is that of subexponential distributions
\cite[cf.][]{FossKorshunovZachary2011}. 
Such a class of distributions contains practically all the heavy-tailed distributions commonly used.
We say that $H$ belongs to the class of subexponential distributions, denoted $H\in\mathcal{S}$, if
\begin{equation*}
  \limsup_{s\to\infty}\frac{\Htail^{*n}(s)}{\Htail(s)}=n,
\end{equation*}
where $\Htail^{*n}$ is the tail probability of the $n$-fold convolution of $H$.

Another important subclass of subexponential distributions that is widely applied in actuarial sciences 
is the class of regularly varying distributions.
A distribution $H$ is regularly varying with index $\alpha>0$ if
\begin{equation}
\lim_{s\to\infty}\frac{\Htail(st)}{\Htail(s)}=t^{-\alpha},\quad t>0,
\end{equation}
and it is denoted $H\in\mathcal{R}_{-\alpha}$.
Otherwise, if the limit above is 
$0$ for all $t>1$, then we say that $H$ is a distribution of \emph{rapid variation} 
and it is denoted $H\in\mathcal{R}_{-\infty}$ \citep[cf.][]{Bingham1987}.
%Among these a very important one is the \emph{principle of the single large jump}, which in heuristic
%terms says that the event where a sum of $n$ iid subexponential random variables exceeds a large
%threshold is \emph{most} likely a consequence of having a single random variable taking a very large
%value (jump).  

\subsection{Phase-type scale mixtures}\label{PHM}
Next we introduce the class of phase-type scale mixture distributions which is central
for this paper.
%We are now ready for introducing the class of distributions studied in this paper.
%\begin{definition}(Phase-type scale mixture distributions).\\
We say a distribution $F(x)$ is a \emph{phase-type scale mixture}  with scaling distribution $H$
and phase-type distribution $G\sim \mathrm{PH}(\boldsymbol{\beta},\mathbf{\Lambda})$, 
if the distribution $F$ can be written as the Mellin--Stieltjes
convolution of $H$ and $G$ (see equation \eqref{PHSM} for a definition).
%\begin{equation}\label{PHSM1}
%F(x)=\int_{0}^{\infty}G(x/s)\dd H(s).
%\end{equation}
%\end{definition}
%Observe that in the definition above 
For this definition to be valid, it is implicit that $H$ must be nonnegative without
an atom at $0$.
Particularly, when the scaling 
distribution  $H$ is discrete and supported over a countable set of nonnegative numbers
$\{s_i:i\in\nat\}$, then the Mellin--Stieltjes convolution in \eqref{PHSM}
reduces to the following infinite series:
\begin{equation*}
F(x)=\sum_{i=1}^{\infty}p(i) G(x/s_i),
\end{equation*}
where $p(i):=H(s_i)-H(s_{i-1})$ is the probability 
mass function of $H$ with $s_0=0$. 
It is not difficult to see that a phase-type scale mixture distribution is absolutely continuous
%\begin{Proposition}\label{Continuity}
% A phase-type scale mixture distribution $F$ is absolutely continuous
 and its density function can be written as
 \begin{align*}
f(x)=\int_0^\infty \frac{g(x/s)}s\dd H(s),
% &=\sum_{j=1}^{m}\sum_{k=0}^{\eta_j-1}c_{jk}
% \int_0^{\infty}\left(\frac{x}{s}\right)^{k}\e^{-\lambda_j x/s}\bigg(\frac{\lambda_j}s-\frac kx\bigg)\dd H(s)\\
% &=\sum_{j=1}^{m}\sum_{k=0}^{\eta_j-1}c_{jk}\,\frac1x\bigg[
% \lambda_j\int_0^{\infty}\left(\frac{x}{s}\right)^{k+1}\e^{-\lambda_j x/s}\dd H(s)
% -k\int_0^{\infty}\left(\frac{x}{s}\right)^{k}\e^{-\lambda_j x/s}\dd H(s)\bigg]
 \end{align*}
 where $g$ is the density of the phase-type distribution.
%\end{Proposition}
%\begin{proof}
% Since $H$ is a finite measure and $G_s(x):=G(x/s)$ is bounded and continuous
% at $x$, then 
% %theorem 16.8 in \cite{Billingsley1995} implies that 
% $F$ is an 
% absolutely continuous distribution with
% \begin{equation*}
%  f(x)=F^\prime(x)=\int_0^\infty \frac{\dd}{\dd x}G(x/s)\dd H(s)
%  =\int_0^\infty \frac{g(x/s)}s\dd H(s).
% \end{equation*}
%\end{proof}
The tail probability of a phase-type scale mixture   
$\Ftail:=1-F$ can also be written as a Mellin--Stiltjes convolution of $H$ and $\overline G$:
\begin{align*}
\Ftail(x)%=&1-F(x)\\
               =1-\int_{0}^{\infty}G(x/s)\dd H(s)
               =\int_0^{\infty}(1-G(x/s))\dd H(s)
               =\int_0^{\infty}\Gtail(x/s)\dd H(s).
\end{align*}
Therefore, using proposition \ref{myprop2.1} it is straightforward to see that there exist
constants $c_{jk}^{\prime}$ and $c_{k}^{\prime}$, such that
\begin{align*}
\Ftail(x)&\le
\sum_{j=1}^{m}\sum_{k=0}^{\eta_j-1}c_{jk}^{\prime}
\int_0^{\infty}\left(\frac{x}{s}\right)^{k}\e^{\Re(-\lambda_j) x/s}\dd H(s)\le \sum_{k=0}^{\eta-1}c_{k}^{\prime}\int_0^{\infty}\left(\frac{x}{s}\right)^{k}\e^{-\lambda x/s}\dd H(s).
 \label{FullTerms}
\end{align*}
%In the expression above, $-\lambda$ is the largest eigenvalue of the sub-intensity matrix $\boldsymbol{\Lambda}$ and
%$\eta$ the largest dimension of the Jordan blocks having associated eigenvalue $-\lambda$.
Hence, only the largest real eigenvalue determines the asymptotic behavior
of a phase-type scale mixture distribution.

In this paper, we are particularly interested in providing sufficient conditions for
a phase-type scale mixture to be subexponential. However, the task of determining whether
a given heavy-tailed distributions is subexponential or not can be very challenging. 
We will resort to extreme value theory to address this issue,
since there exist a variety of results relating the subexponential property with
maximum domains of attraction. 
%First, we state the central theorem in extreme value theory \citep{Fisher1928,Gnedenko1943}.
%\begin{theorem}[Fisher-Tippett-Gnedenko Theorem]
%\label{FRG}
%Let $\{S_1,\cdots S_n\}$ be a sequence of i.i.d. random variables 
%with common distribution $H$. 
%If there exist some proper sequences of (norming)
%constants $c_n>0$ and $d_n\in\RL$ and non-degenerate distribution $M$ 
%such that
%\begin{equation*}
%\dfrac{\max\{S_1,\cdots,S_n\}-d_n}{c_n}\stackrel{d}{\longrightarrow}S,\quad \text{where}\ S\sim M,
%\end{equation*}
%then $M$ can only belong to one of three distributions: Fr\'echet ($\Phi$), 
%Weibull ($\Psi$)
%or Gumbel ($\Lambda$). 
%We say that $H$ belongs to the maximum domain attraction of $M$: $H\in$MDA($M$)
%and we call $c_n$ and $d_n$ the sequence of norming constants.
%\end{theorem}
%In this Section we first recall several fundamental concepts needed for the development of
%our results.  We discuss heavy-tailed distributions, 
%maximum domains of attraction and subexponentiality.  
%We revise the asymptotic properties 
%of phase-type distributions.
%Then we introduce the class of phase-type scale mixtures which
%is the main object of study in this paper. 
%Finally, we collect the relevant results of the tail behavior of the products.

The Weibull domain of attraction is composed
of distributions with support bounded above, so a phase-type scale mixture 
cannot belong to such domain.
The Fr\'echet domain of attraction is characterized by \emph{regular variation} \citep{Haan1970}:

%Regular variation characterizes the Fr\'echet domain of attraction via the following relation
 \begin{equation*}
\Htail\in\mathcal{R}_{-\alpha}\quad\Longleftrightarrow\quad H\in\text{MDA}(\Phi_\alpha).
\end{equation*}
This characterisation is relevant to us because regularly varying distributions are
subexponential.
The Gumbel domain of attraction is more involved. It contains both light- and 
heavy-tailed distributions. A number of results exist for determining the 
Gumbel domain of attraction and subexponentiality of a certain distribution. 
We have listed these in the Appendix since these will be used later.
%This will be the material of the following section.

\section{Tail behavior of scaled random variables}
\label{mysec3}
This section is devoted to characterising the tail properties of the 
class of phase-type scale mixture distributions.  Firstly, we collect some relevant 
results about the asymptotic tail behavior of products of random variables, which provide
sufficient conditions on the scaling random variable $S$ for its associated
phase-type scale mixture distribution to be either light- or heavy-tailed. In addition, we extend
this result to provide a criteria for  more general distributions; we also provide a simplified proof (Theorem \ref{Products}).

Secondly, in Subsection \ref{MDA} we focus on determining the maximum domain of attraction 
of a phase-type scale mixture distribution according to its scaling distribution. 
In the Fr\'echet case, Breiman's 
lemma implies that a phase-type scale mixture distribution remains in the Fr\'echet domain
of attraction (hence regularly varying) if the scaling distribution is in the same domain. 
The converse of Breiman's lemma does not hold true in general, and finding sufficient conditions
and counterexamples is considered challenging \citep[cf.][]{MikoschJacobsenRosinski2009,DenisovZwart2005,MikoschJessen2006,Damek2014}.
For the Gumbel case, we provide conditions on the Laplace transform of reciprocal of 
the scaling random variable $1/S$ so the associated phase-type scale mixture distribution belongs to the Gumbel domain 
of attraction, as well as to further determine if it is subexponential. We illustrate with examples that such conditions are verifiable in some important cases. 
In addition, we also analyse the important class of Weibullian distributions (for a definition
see Remark \ref{Weibullian} below)
which posseses a closure property under multiplication \citep{ArendarczykDebicki2011}.
The result in that paper allows to determine the exact tail behavior of a phase-type
scale mixture having a Weibullian scaling distribution.

\subsection{Asymptotic tail behavior}
The tail behavior of the distribution of a product of nonnegative random variables 
has attracted a considerable amount of research interest. For instance,
\cite{SuChen2006} show that if two random variables $S_1$ and $S_2$ are such 
that the distribution of $S_1$ is in $\mathcal{L}(\lambda)$ with $\lambda>0$
and $S_2$ has unbounded support, then the distribution of $S_1\cdot S_2$ is in $\mathcal{L}(0)$ (long-tailed),
and thus heavy-tailed
\citep[see also][]{Tang2008b}.  If one further assumes that $S_2$ is Weibullian 
with parameter $0<p\le 1$,
 then \cite{LiuTang2010} show that the product $S_1\cdot S_2$ is subexponential.
A result which extends beyond the class $\mathcal{L}(\gamma)$ is in \cite{ArendarczykDebicki2011}, 
where it is shown that the product of two Weibullian random variables with parameters $p_1$ and $p_2$ is Weibullian
with parameter $p_1p_2/(p_1+p_2)$ and thus proving that  the product of Weibullians 
can be either light- or heavy-tailed.

These results imply that a phase-type scale mixture distribution is heavy-tailed
if and only if the scaling distribution has unbounded support.
This conclusion can also be obtained from our Theorem \ref{Products} below, where we provide
sufficient conditions under which a product of two general 
random variables can be classified either as light- or heavy-tailed. 
The simplified proof provided here is elementary.

\begin{theorem}\label{Products}
  Consider $S_1$ and $S_2$ two nonnegative independent random variables with unbounded support, 
  where $S_1\sim H_1$ and $S_2\sim H_2$. 
  Let $H$ be the distribution of the product $S_1\cdot S_2$.  
  \begin{enumerate}
   \item If there exist $\theta>0$ and $\xi(x)$ a nonnegative function such that
     \begin{equation}\label{hypo1}
      \limsup_{x\to\infty}\e^{\theta x}\Big(\Htail_1(x/\xi(x))+\Htail_2(\xi(x))\Big)=0,
     \end{equation}
     then $H$ is a light-tailed distribution.
    \item If there exists $\xi(x)$ a nonnegative function such that for all $\theta>0$ it holds that
    \begin{equation}\label{hypo2}
     \limsup_{x\to\infty}\ \e^{\theta x}\,\Htail_1(x/\xi(x))\cdot\Htail_2(\xi(x))=\infty,
    \end{equation}
    then $H$ is a heavy-tailed distribution.
  \end{enumerate}
\end{theorem}

\begin{proof}
For the first part consider
\begin{align*}
  \limsup_{x\to\infty}\Htail(x)\e^{\theta x}
   &=\limsup_{x\to\infty}\e^{\theta x}\int_{0}^{\infty} \Htail_1(x/s)\dd H_2(s)\\
   &=\limsup_{x\to\infty}\bigg[\e^{\theta x}\int_{0}^{\xi(x)} \Htail_1(x/s)\dd H_2(s)+
    \e^{\theta x}\int_{\xi(x)}^\infty \Htail_1(x/s)\dd H_2(s)\bigg]\\
   &\le\limsup_{x\to\infty}\left[\e^{\theta x}\Htail_1(x/\xi(x))+
    \e^{\theta x} \Htail_2(\xi(x))\right]=0.
 \end{align*}
The last equality holds by the hypothesis \eqref{hypo1}. Hence $H$ is light-tailed.
For the second part consider
\begin{align*}
 \limsup_{x\to\infty}\Htail(x)\e^{\theta x}
%    &=\limsup_{x\to\infty}\bigg[\e^{\theta x}\int_{0}^{\infty} \Htail_1(x/s)\dd H_2(y)\\
   &=\limsup_{x\to\infty}\bigg[\e^{\theta x}\int_{0}^{\xi(x)} \Htail_1(x/s)\dd H_2(s)+
    \e^{\theta x}\int_{\xi(x)}^\infty \Htail_1(x/s)\dd H_2(s)\bigg]\\
%    &\ge\limsup_{x\to\infty}\e^{\theta x}\Htail_1(x/\xi(x))\int_{\xi(x)}^{\infty}\dd H_2(x)\\
   &\ge \limsup_{x\to\infty}\left[\e^{\theta x}\Htail_1(x/\xi(x))\Htail_2(\xi(x))\right]=\infty.
 \end{align*}
The last equality holds by hypothesis \eqref{hypo2}. Hence $H$ is heavy-tailed.
\end{proof}

The conditions in Theorem \ref{Products} can be easily verified 
and  enables us to provide a classification of the asymptotic tail behavior of products of
random variables with more general distributions. 
Notice that the distributions considered in \cite{SuChen2006} correspond to distributions 
with log-tail probabilities decaying at a linear rate, i.e. $-\log\overline{H_1}(s)=O(s)$,
while the distributions in \cite{ArendarczykDebicki2011} have log-tail probabilities
decaying at a power rate, 
i.e. $-\log\overline{H_i}(s)=O(s^{p_i})$, $i=1,2$.  
The following example considers distributions with log-tail probabilities decaying  
at an exponential rate, i.e.\ $-\log\overline{H_i}(s)=O(\e^{s})$.
\begin{Example}[Gumbellian products]
Let $H_i(x)=1-\exp\{-\e^{x}+1\}$, $x>0$. We choose $\xi(x)=x^{\gamma}$, with $0<\gamma<1$. Then 
\begin{equation*}
\lim_{x\to\infty}\Htail(x)\e^{\theta x}=\lim_{x\to\infty}\e^{\theta x+1}\left(\exp\{-\e^{x^{1-\gamma}}\}+\exp\{-\e^{x^{\gamma}}\}\right)=0,\quad \forall \theta>0.
\end{equation*}
Then the product of two random variables with Gumbellian-type distributions is always light-tailed.
The same holds true if we replace $H_2$ with a Weibullian distribution with shape parameter $p>1$. 
Choose $\xi(x)=x^{\gamma}$, with $\frac1p\leq\gamma<1$ and observe that
\begin{equation*}
\lim_{x\to\infty}\Htail(x)\e^{\theta x}=\lim_{x\to\infty}\e^{\theta x}\left(\exp\{-\e^{x^{1-\gamma}}+1\}+x^{\delta}\e^{-x^{\gamma p}}\right)=0,\quad \text{for}\ \theta\in(0,1).
\end{equation*}
%Thus, the product of Gumbellian and Weibullian with shape parameter $p>1$ is light-tailed.
\end{Example}

%\begin{Example}[Geometric Products]
%Let $H_1\sim$Geo($p$) and $H_2\sim$Geo($q$) supported over \{$0,1,\cdots$\}.
%Choose $\xi(x)=\sqrt{x}$, then
%\begin{align*}
%&\lim_{x\to\infty}\e^{\theta x}\Htail_1(x/\xi(x))\cdot\Htail_2(\xi(x))
%=\lim_{x\to\infty}\e^{\theta x}\Htail_1(\sqrt{x})\cdot\Htail_2(\sqrt{x})\\
%=&pq\lim_{x\to\infty}\exp\{\theta x+(\ln(1-p)+\ln(1-q))\lfloor{\sqrt{x}}\rfloor \}=\infty,\quad \forall \theta>0,
%\end{align*}
%where $\lfloor{x}\rfloor$ is the floor function of $x$.
%Then the product of two geometric random variables has a discrete heavy-tailed distribution.
%
%
%Replace $H_1\sim$Weibull($\lambda,p$) and choose $\xi(x)=x^{\gamma}$, where $\gamma\in(1-p^{-1},1)$, then
%\begin{equation*}
%\lim_{x\to\infty}\e^{\theta x}\Htail_1(x/\xi(x))\Htail_2(\xi(x))
% =\lim_{x\to\infty}\left(\e^{\theta x-(x^{1-\gamma}/\lambda)^p+\ln(1-\beta)\lfloor{x}\rfloor^{\gamma}}\right)=\infty,\qquad \forall\theta>0.
%\end{equation*}
%The product of geometric and any Weibull distribution is heavy-tailed.
%\end{Example}

\subsection{Maximum domains of attraction and subexponentiality}
\label{MDA}
The scenario in the Fr\'echet domain of attraction is well understood.
Breiman's lemma 
\citep{Breiman1965} implies that a phase-type scale mixture distribution 
is in the Fr\'echet domain of attraction if its scaling distribution is in the same domain:
\begin{Lemma}[\cite{Breiman1965}]
\label{Breiman}
If $H\in\mathcal{R}_{-\alpha}$
and $M_G(\alpha+\epsilon)<\infty$ for some $\epsilon>0$, then $F\in\mathcal{R}_{-\alpha}$
and  
\begin{equation}
  \Ftail(x) = M_G(\alpha)\Htail(x) (1+\oh(1)),\qquad x\to\infty,
\end{equation}
where $M_G(\alpha)$ is the $\alpha$-moment of $G$.
\end{Lemma}
Phase-type distributions
are light-tailed so all their moments are finite. Therefore,
a phase-type scale mixture distribution with a scaling distribution 
in the Fr\'echet domain of attraction remains in the same domain.
% \begin{equation*}
% H\in\mathrm{MDA}(\Phi_\alpha) \Longrightarrow F\in\mathrm{MDA}(\Phi_\alpha),
%\end{equation*}
%where $F$ is the phase-type scale mixture distribution. Furthermore, the norming constants can be chosen as
%\begin{Corollary} 
%Assume that $H$ is regularly varying with index $-\alpha$, then the norming constants can be chosen as 
Furthermore, the norming constants for a phase-type scale mixture distribution $F$
can be chosen as the norming constants of $H$ divided by the $\alpha$-moment
of the phase-type distribution $G$, that is
\begin{equation*}
d_n=0, \quad c_n=\dfrac{1}{M_G(\alpha)}\left(\dfrac{1}{\Htail}\right)^{\leftarrow}(n).
\end{equation*}
Moreover, when the conditions of Breiman's lemma are satisfied, then the scaling
and the phase-type scale mixture distributions
are regularly varying with the same index of regular variation, thus implying
that the tail probabilities of both distributions are asymptotically proportional (with the reciprocal of the
$\alpha$-moment of the phase-type distribution being the proportionality constant).
This implies that the class of phase-type scale mixture distributions can provide
exact asymptotic approximations of the tail probabilities of  regularly varying distributions.

It is interesting to note that the converse of Breiman's lemma does not hold true in general. 
Such a problem is considered 
to be challenging and has attracted considerable research interest, thus resulting in a rich
variety of results proving sufficient conditions and counterexamples; 
for instance, \cite{MikoschJessen2006} provide a comprehensive list of earlier references; 
the most general results are given in \cite{MikoschJacobsenRosinski2009} and \cite{DenisovZwart2005} 
(see also \cite{Damek2014} for a multivariate version).
%\correction{That is, we wish to determine if having a scaling 
%distribution in $\mathcal{R}_{-\alpha}$
%is a necessary condition for a phase-type scale mixture to be in the Fr\'echet 
%domain of attraction. The following counterexample shows that this is not the case.}
It is not difficult to verify that some subclasses (for instance, exponential, Erlang and hyperexponential) 
of PH distributions 
satisfy the sufficient conditions for the converse of Breiman's lemma provided 
in \cite{MikoschJacobsenRosinski2009}. We also conjecture that in general PH distributions satisfy the above conditions but a proof remains unknown to us.

The situation is less understood in the Gumbel domain of attraction.  
We start by noting that in the Gumbel case, 
a phase-type scale mixture $F$ and its scaling distribution $H$
will have very different tail behaviors (this is contrast to the Fr\'echet
case where Breiman's lemma implies that these have asymptotically proportional tail
behavior).  
In particular, the tail probability of a scaling distribution in the Gumbel 
domain of attraction is tail equivalent to a von Mises functions, hence rapidly varying. 
In such a case the tail distribution of the phase-type scale mixture will be much
heavier than its scaling distribution:
\begin{Proposition}
If  $H\in\mathcal{R}_{-\infty}$, then
\begin{equation}\label{tails differ}
 \limsup_{x\to\infty}\frac{\Htail(x)}{\Ftail(x)}=0.
\end{equation}
\end{Proposition}
\begin{proof} 
To show this we take $t>1$ and observe that there exists a constant $C$ such that
\begin{align*}
\Ftail(x)=\Prob[SY>x]\geq \Prob[SY>x,Y
\geq t]\geq&\Prob[S>x/t]\Prob[Y\geq t]=\Htail(x/t)C,
\end{align*}
Then
\begin{equation*}
\limsup_{x\to\infty}\dfrac{\Htail(x)}{\Ftail(x)}\leq\dfrac{1}{C}\limsup_{x\to\infty}\dfrac{\Htail(x)}{\Htail(x/t)}=0,\quad t>1.
\end{equation*} 
\end{proof}
The lognormal and Weibullian distributions are rapidly varying.  
\begin{Remark}
This result fleshes out a limitation of the aforementioned approach for approximating distributions in
the Gumbel domain of attraction. The tail probability of a phase-type scale mixture 
distribution will be much heavier than its target distribution, if the 
scaling distribution is chosen within the same family of target distributions and with similar parameters.
We show later that in some cases we are able to construct phase-type scale mixture distributions 
with the same asymptotic behavior as their target distributions if
we vary the value of parameters.  Such is the case of Weibullian distributions.
\end{Remark}

Next we look for sufficient conditions of the scaling distribution so its corresponding
phase-type scale mixture will belong to the Gumbel domain of attraction and be
subexponential.
We restrict our focus to phase-type distributions with
sub-intensity matrices having only real eigenvalues.

\begin{theorem}
\label{Lemma.2}
%Assume $F$, $G$ and $H$ satisfy the assumptions of Lemma
%\ref{Breiman}.
Let $V(x)=(-1)^{\eta-1}\mathcal{L}_{1/S}^{(\eta-1)}(x)$ where
$\eta$ is the largest dimension among the Jordan blocks associated to the largest
eigenvalue of the sub-intensity matrix.
If $V(\cdot)$ is a von Mises function, then $F\in\mathrm{MDA}(\Lambda)$.
Moreover, $F$ is subexponential if
\begin{equation*}
 \liminf_{x\to\infty}\frac{V(tx)V^{\prime}(x)}{V^{\prime}(tx)V(x)}>1,\qquad \forall t>1.
\end{equation*}

\end{theorem}

\begin{proof}
%Recall that $F\in\mathrm{MDA}(\Lambda)$ iff
%\begin{equation}\label{FvonMises}
 %\lim_{x\to\infty}\frac{\Ftail(x)F^{\prime\prime}(x)}{(F^\prime(x))^2}=-1
%\end{equation}
%\citep{Haan1970}.  In addition, let $a(x)=\Ftail(x)/F^{\prime}(x)$ which is called
%an \emph{auxiliary function of $F$}, then if
%\begin{equation}\label{Fsub}
% \liminf_{x\to\infty}\frac{a(tx)}{a(x)}>1,\qquad \forall t>1,
%\end{equation}
%then $F$ is subexponential \citep{GoldieResnick1988}. Hence, we shall prove that
%the hypotheses of the theorem \ref{Lemma.2} imply
%\eqref{FvonMises} and \eqref{Fsub}.

We can write that
\begin{equation*}
\Ftail(x)=
\sum_{j=1}^{m}\sum_{k=0}^{\eta_j-1}
\int_0^{\infty}c_{jk}\left(\frac{x}{s}\right)^k\e^{-\lambda_j x/s}\dd H(s)
=\sum_{j=1}^{m}\sum_{k=0}^{\eta_j-1} c_{jk}\, \frac{(-1)^kx^k}{\lambda_j^k}\,
 \mathcal{L}_{1/S}^{(k)}(\lambda_j x).
\end{equation*}

 %we can write
%\begin{equation*}
 %\Ftail(x)=\sum_{j=1}^{m}\sum_{k=0}^{\eta_j-1} c_{jk}\, \frac{(-1)^kx^k}{\lambda_j^k}\,
 %\mathcal{L}_{1/S}^{(k)}(\lambda_j x).
%\end{equation*}
Since $V(x)=(-1)^{\eta-1}\mathcal{L}_{1/S}^{(\eta-1)}(x)$ is a von Mises function, then
$V(x)$ is of rapid variation \citep{Bingham1987}.  
This implies that
\begin{equation}\label{FtailGumbel}
 \Ftail(x)\sim c\frac{x^{\eta-1}}{\lambda^{\eta-1}}V(\lambda x),
%  \frac{\dd^{\eta-1}}{\dd x^{\eta-1}}\mathcal{L}_{1/S}(\lambda x),
\end{equation}
where $c$ is some constant, $-\lambda$ is the largest eigenvalue of the 
sub-intensity matrix and $\eta$ is the largest dimension among the Jordan blocks
associated to $-\lambda$.
%In addition
%\begin{align*}
% F^{\prime}(x)&=-\sum_{j=1}^{m}\sum_{k=0}^{\eta_j-1} c_{jk}\, 
% \bigg[\frac{(-1)^{k}kx^{k-1}}{\lambda_j^k}\mathcal{L}_{1/S}^{(k)}(\lambda_j x)
%  +\frac{(-1)^{k}x^k}{\lambda_j^{k-1}}\mathcal{L}_{1/S}^{(k+1)}(\lambda_j x)\bigg]\\
%  &\sim -c\frac{(-1)^{\eta-1}x^{\eta-1}}{\lambda^{\eta-2}}\mathcal{L}_{1/S}^{(\eta)}(\lambda x)=-c\dfrac{x^{\eta-1}}{\lambda^{\eta-2}}V^{\prime}(\lambda x),
%% \frac{\dd^{\eta}}{\dd x^{\eta}}\mathcal{L}_{1/S}(\lambda x)
%\end{align*}
%and
%\small{
%\begin{align*}
% F^{\prime\prime}(x)&=-\sum_{j=1}^{m}\sum_{k=0}^{\eta_j-1} c_{jk}\, 
% \bigg[\frac{(-1)^{k}k(k-1)x^{k-2}}{\lambda_j^k}\mathcal{L}_{1/S}^{(k)}(\lambda_j x)
%  +\frac{(-1)^{k}kx^{k-1}}{\lambda_j^{k-1}}\mathcal{L}_{1/S}^{(k+1)}(\lambda_j x)\\
% &\qquad\qquad+\frac{(-1)^{k}kx^{k-1}}{\lambda_j^{k-1}}\mathcal{L}_{1/S}^{(k+1)}(\lambda_j x)
%  +\frac{(-1)^{k}x^{k}}{\lambda_j^{k-2}}
% \mathcal{L}_{1/S}^{(k+2)}(\lambda_j x)\bigg]\\
%  &\sim -c\frac{(-1)^{\eta-1}x^{\eta-1}}{\lambda^{\eta-3}}\mathcal{L}_{1/S}^{(\eta+1)}(\lambda x)=-c\dfrac{x^{\eta-1}}{\lambda^{\eta-3}}V^{\prime\prime}(\lambda x).
%% \frac{\dd^{\eta+1}}{\dd x^{\eta+1}}\mathcal{L}_{1/S}(\lambda x).
%\end{align*}
%}
%Hence it follows that 
Then it is not difficult to see that
\begin{equation*}
 \lim_{x\to\infty}\frac{\Ftail(x)F^{\prime\prime}(x)}{(F^\prime(x))^2}=
 \lim_{x\to\infty}\frac{V(\lambda x)(-V^{\prime\prime}(\lambda x))}
  {\big(-V^{\prime}(\lambda x)\big)^2}=-1.
\end{equation*}
This holds true because by hypothesis $V(x)=(-1)^{\eta-1}\mathcal{L}_{1/S}^{(\eta-1)}(x)$ is a von Mises function.  Hence $F\in\mathrm{MDA}(\Lambda)$
and the first part result follows.  For the second part, we observe that
the auxiliary function $a(x)=\Ftail(x)/F^{\prime}(x)$
obeys the following asymptotic equivalence
\begin{equation*}
 a(x)=\frac{\Ftail(x)}{F^{\prime}(x)}\sim\frac{V(\lambda x)}{-\lambda V^{\prime}(\lambda x)}.
\end{equation*}
The distribution $F$ is subexponential if
\begin{equation*}
 \liminf_{x\to\infty}\frac{a(tx)}{a(x)}=
  \liminf_{x\to\infty}\frac{V(\lambda tx)V^{\prime}(\lambda x)}{V^{\prime}(\lambda tx)V(\lambda x)}>1,\qquad \forall t>1,
\end{equation*}
hence subexponentiality of $F$ follows.

\end{proof}

Theorem \ref{Lemma.2} can be applied to the lognormal case:
\begin{Example}[Lognormal scaling]
\label{myex lognormal}
Assume $H\sim\mathrm{LN}(\mu,\sigma^2$), then $F$ is a subexponential distribution
in the Gumbel domain of attraction.
\end{Example}
\begin{proof}
W.l.o.g.\ we can assume $\mu=0$ since $\e^\mu$ is a scaling constant.
In such a case the symmetry of the normal distribution implies that the Laplace--Stieltjes transform of 
$1/S$ is the same as that of $S$, i.e.
\begin{equation*}
\mathcal{L}_{1/S}(x)=\mathcal{L}_{S}(x).
\end{equation*}

An asymptotic approximation of the $k$-th derivative of the Laplace--Stieltjes transform of the lognormal distribution is given in \cite{Asmussen2014}:
\begin{equation*}
\mathcal{L}_{S}^{(k)}(x)=(-1)^k\mathcal{L}_S(x)\exp\{-k\omega_0(x)+\dfrac{1}{2}\sigma_0(x)^2k^2\}(1+\oh(1)),
\end{equation*}
where
\begin{equation*}
\omega_k(x)=\mathcal{W}(x\sigma^2\e^{k\sigma^2}),\quad \sigma_k(x)^2=\dfrac{\sigma^2}{1+\omega_k(x)},
\end{equation*}
and $\mathcal{W}(\cdot)$ is the Lambert W function.
Hence we verify that
\begin{align*}
\lim_{x\to\infty}\dfrac{V(x)(-V^{\prime\prime}(x))}{(-V^{\prime}(x))^2}
&=\lim_{x\to\infty}\dfrac{\e^{-(\eta-1)\omega_0(x)+\frac{1}{2}\sigma_0(x)^2(\eta-1)^2}\cdot\left(-\e^{-(\eta+1)\omega_0(x)+\frac{1}{2}\sigma_0(x)^2(\eta+1)^2}\right)}{\e^{-2\eta\omega_0(x)+\sigma_0(x)^2\eta^2}}\\
&=-\lim_{x\to\infty}\exp\{\sigma_0(x)^2\}=-\lim_{x\to\infty}\exp\left\{\dfrac{\sigma^2}{1+\omega_0(x)}\right\}.
\end{align*}
As $\omega_k(x)$ is asymptotically of order $\log(x)$ as $x\to\infty$, then ${\sigma^2}{(1+\omega_0(x))^{-1}}\to 0$ as $x\to\infty$. Then 
the last limit is equal to $-1$,
so we have shown that $F(x)\in\mathrm{MDA}(\Lambda)$.
Furthermore, 
\begin{align*}
\lim_{x\to\infty}\dfrac{a(tx)}{a(x)}&=\lim_{x\to\infty}\dfrac{(-1)^{\eta-1}\mathcal{L}^{(\eta-1)}_{1/S}(tx)\cdot(-1)^{\eta-1}\mathcal{L}^{(\eta)}_{1/S}(x)}{(-1)^{\eta-1}\mathcal{L}^{(\eta)}_{1/S}(tx)\cdot(-1)^{\eta-1}\mathcal{L}^{(\eta-1)}_{1/S}(x)}\\
&=\lim_{x\to\infty}\dfrac{\e^{-(\eta-1)\omega_0(xt)+\frac{1}{2}\sigma_0(xt)^2(\eta-1)^2}\cdot \e^{-\eta\omega_0(x)+\frac{1}{2}\sigma_0(x)^2\eta^2}}{\e^{-\eta\omega_0(xt)+\frac{1}{2}\sigma_0(xt)^2\eta^2}\cdot \e^{-(\eta-1)\omega_0(x)+\frac{1}{2}\sigma_0(x)^2(\eta-1)^2}}\\
&=\lim_{x\to\infty}\exp\left\{-\omega_0(x)+\omega_0(xt)+\dfrac{1}{2}\sigma_0(xt)^2(2\eta-1)+\dfrac{1}{2}\sigma_0(x)^2(1-2\eta)\right\}\\
&=\lim_{x\to\infty}\exp\left\{-\omega_0(x)+\omega_0(x)+\omega_0(t)+\Oh(\omega_0(x)^{-1})\right\}=t>1.
\end{align*}
Thus $F$ is a subexponential distribution.
\end{proof}

\begin{Example}[Exponential scaling]
\label{myex4}
Let $H\sim\exp(\beta)$. Then $F$ is a subexponential distribution in
 the Gumbel domain of attraction.
\end{Example}
\begin{proof}
Observe that $1/S$ has an inverse gamma distribution with a Laplace--Stieltjes transform 
given in terms of a modified Bessel function of the second kind \citep{Raqab1965}:
\begin{equation*}
\mathcal{L}_{1/S}(x)=\int_0^{\infty}\e^{- x/s}\beta\e^{-\beta s}\dd s=2\sqrt{\beta x}\text{BesselK}(1, 2\sqrt{\beta x}).
\end{equation*}
Furthermore, its $n$-th derivative can be calculated explicitly also in terms of a modified Bessel function of the second kind:
\begin{equation*}
\mathcal{L}_{1/S}^{(n)}(x)=\int_0^{\infty}\left(-\frac{1}{s}\right)^n\e^{-x/s}\beta\e^{-\beta s}\dd s=(-1)^n\cdot 2\,\beta^{\frac{n+1}{2}}x^{-\frac{n-1}{2}}\text{BesselK}(n-1, 2\sqrt{\beta x}).
\end{equation*}
Asymptotically it holds true that 
\begin{equation*}
\mathcal{L}_{1/S}^{(n)}(x)\sim (-1)^n\sqrt{\pi}\beta^{\frac{2n+1}{4}} x^{-\frac{2n-1}{4}}\e^{-2\sqrt{\beta x}},\quad x\to\infty.
\end{equation*}
Hence, it follows that
\begin{equation*}
 \lim_{x\to\infty}\dfrac{V(x)(-V^{\prime\prime}(x))}{(-V^{\prime}(x))^2}
% \to\dfrac{(-1)^{\eta-1}\text{BesselK}(\sqrt{\beta\lambda x})(-1)^{\eta+1}\text{BesselK}(\sqrt{\beta\lambda x})}{((-1)^{\eta}\text{BesselK}(\sqrt{\beta\lambda x}))^2}=1.
=-1.
\end{equation*}
Therefore, $V(x)$ is a von Mises function and $F\in\mathrm{MDA}(\Lambda)$. 
Moreover, if $t>1$ then
\begin{align*}
\lim_{x\to\infty}\dfrac{a(tx)}{a(x)}=\lim_{x\to\infty}\dfrac{V(tx)V^{\prime}(x)}{V^{\prime}(tx)V(x)}=\sqrt{t}>1.
\end{align*}
Thus F is a subexponential distribution.

\end{proof}

\begin{Remark}
 Notice that it is possible to generalize the result of the previous example for a
 gamma scaling distribution, because an expression
 for the Laplace--Stieltjes transform of an inverse gamma distribution is known and
 given in terms of a modified Bessel function of the second kind.  However, it involves 
 a number of tedious calculations and therefore omitted.
 Note as well that in such a case it is possible to test directly if $\Ftail$
 is a von Mises function, but the calculations become cumbersome.  
 Finally, we remark that the results of \cite{LiuTang2010} imply the
 subexponentiality of the exponential case.
\end{Remark}

\begin{Remark}
If $H$ is a discrete scaling distribution, then we can obtain an analogue result to that of 
of Theorem \ref{Lemma.2}. Define 
\begin{equation*}
\mathcal{DL}_{1/S}(x)=\sum_{i=1}^{\infty}\e^{-x/i}p(i)
\end{equation*}
as the Laplace--Stieltjes transform of discrete scaling random variable 
$S$ with probability mass function $p(i)$.
Then the tail probability of the phase-type scale mixture is:
\begin{equation*}
\Ftail(x)=
\sum_{j=1}^{m}\sum_{k=0}^{\eta_j-1}
\sum_{i=1}^{\infty}c_{jk}\left(\frac{x}{i}\right)^k\e^{-\lambda_j x/i}p(i)
=\sum_{j=1}^{m}\sum_{k=0}^{\eta_j-1} c_{jk}\, \frac{(-1)^kx^k}{\lambda_j^k}\,
 \mathcal{DL}_{1/S}^{(k)}(\lambda_j x).
\end{equation*}
If $V(x)=(-1)^{\eta-1}\mathcal{DL}_{1/S}^{(\eta-1)}(x)$ is a von Mises function,
then $F\in\mbox{MDA}(\Lambda$).
\end{Remark}

We close this section with an important remark regarding Weibullian scalings.

\begin{Remark}[Weibullian scaling]
\label{Weibullian}
A nonnegative 
distribution $H$ is said to be Weibullian with shape parameter $p>0$ \citep{ArendarczykDebicki2011} if 
\begin{equation*}
\Htail(s)= Cs^{\delta}\exp(-\lambda s^{p})(1+o(1)),\quad  \lambda,C>0, \delta\in\RL.
\end{equation*}
A Weibullian distribution with parameter $p$ is heavy-tailed if $0<p<1$, while
it is light-tailed if $p\ge 1$.
Notice that a phase-type distribution is
Weibullian with shape parameter equal to $1$.  Therefore, Lemma 2.1 of \cite{ArendarczykDebicki2011}
implies that a phase-type scale mixture having 
a Weibullian scaling distribution with scale parameter $p$ will be Weibullian with shape 
parameter ${p_1}(1+p_1)^{-1}<1$, thus heavy-tailed. 
Furthermore, Lemma 2.1 in \cite{ArendarczykDebicki2011} provides exact expressions for each of the parameters 
$C$, $\delta$ and $\lambda$, so in principle one can use this result to replicate 
exactly the tail behavior of a Weibullian distribution via a phase-type scale mixture distribution. 
%The class of Weibullian distributions with parameters $p=1$ is contained in $\mathcal{L}(1)$. 
\end{Remark}

%An example to illustrate Gumbel domain of attraction is complicated.

\section{Discrete scaling distributions}
\label{mysec4}
Next we focus on the case of phase-type scale mixture distributions having 
scaling distributions supported over countable sets of strictly positive numbers.
These distributions are particularly tractable  since these 
correspond to distributions of absorption times of Markov jump processes with an infinite number of transient states.  
This class of distributions is of great importance for applications involving heavy-tailed phenomena,
since  a variety of quantities of interest can be calculated exactly.  
Such is the case of ruin probabilities in the Cr\'amer-Lundberg process 
having claims sizes distributed according to a phase-type scale mixture
\citep[cf.][]{Bladt2014b,Nardo2016}.  Notice for instance, that such exact results
are not available for the case of continuous scaling distributions.

We remark however, that some of the methodologies
for determining domains of attraction and subexponentiality described in the previous section are not always
implementable in a straightforward way for discrete scaling distributions. 
One of the main difficulties is the calculation of asymptotic equivalent expressions for the infinite series defining the 
tail probabilities. Below we describe a simple methodology which can be used to extend results for 
continuous scaling distributions to their discrete scaling distributions counterparts; 
such a methodology provides mild conditions under which the asymptotical behavior of an infinite
series is asymptotically equivalent to that of a certain function defined via a definite integral.

\begin{Proposition}
\label{integral approx}
Let $I_u:\mathbf{Z}^+\to\RL^+$ be collection of functions indexed by $u\in(0,\infty)$.
Suppose that for each $u>0$ there exists a measurable and bounded function $I'_u:\RL^+\to\RL$ such that $I(u;k)=I'(u;k)$ 
for all $k\in\mathbb{Z}^+$ and
\begin{equation*}
\int_{0}^{\infty}I'(u;y)\dd y - M(u)\le\sum_{k=0}^{\infty}I(u;k)\le\int_{0}^{\infty}I'(u;y)\dd y +M(u),
\end{equation*}
where $M(u)\ge\max\{I'(u;y):y>0\}$ is some upper bound for the function $I'(u;y)$.  If
\begin{equation*}
\lim\limits_{u\to\infty}\dfrac{M(u)}{\int_0^{\infty}I'(u;y)\dd y}=0,
\end{equation*}  then the following asymptotic relationship holds
\begin{align*}
\lim_{u\to\infty}\dfrac{\sum_{k=0}^{\infty}I(u;k)}{\int_0^{\infty}I'(u;y)\dd y}=1.
\end{align*}
\end{Proposition}

The method provides a verifiable condition under which the infinite series can be replaced
by an asymptotic integral. The next example is taken from \cite{Bladt2014b}. 
\begin{Example}[Zeta scaling]
\label{Zeta}
Let $\alpha\ge 2$ and assume $H\sim\mathrm{Zeta}(\alpha)$.
Such a distribution is determined by $p(i)=i^{-\alpha}/\zeta(\alpha)$, $i\in\nat$ and
$\zeta(\cdot)$ is the Riemann zeta function. 
Then $F$ is in the Fr\'echet domain of attraction.
\end{Example}
We remark that Breiman's lemma could have been used instead to determine the exact 
asymptotic behavior because the tail probability $\Htail(i)$, $i=1,2,\dots$ 
forms a regularly varying sequence, so $\Htail\in\mathcal{R}_{-\alpha}$ \citep{Bingham1987}. 
Nevertheless, this example is  included here to illustrate the simplicity of the method proposed.
%(an example with the Zipf distribution is given in \citep{Bladt2014b}). 
\begin{proof}
$H$ is supported over all the natural numbers, so the tail probability of corresponding phase-type scale mixture can be written as 
\begin{equation*}
\Ftail(x)=\sum_{i=1}^{\infty}p(i)\Gtail(x/i).
\end{equation*}

Recall that the expression of $\Gtail(\cdot)$ has been given in \eqref{PH tail}, then we have
\begin{align*}
 \Ftail(x)=\sum_{i=1}^{\infty}\sum_{j=1}^{m}\sum_{k=0}^{\eta_j-1}c_{jk}\left(\frac{x}{i}\right)^{k}\e^{-\lambda_j x/i}\frac{i^{-\alpha}}{\zeta(\alpha)}
                =\sum_{j=1}^{m}\sum_{k=0}^{\eta_j-1}\sum_{i=1}^{\infty}\dfrac{c_{jk}x^k}{\zeta(\alpha)}i^{-(\alpha+k)}\e^{-\lambda_j x/i}. 
\end{align*}
% for any fixed $\lambda_j$ and $k$, let $g(i)=i^{-(\alpha+k)}\e^{-\lambda_j x/i}$, and 
Consider the functions $I'_{jk}(x;y)=x^ky^{-(\alpha+k)}\e^{-\lambda_j x/y}$ and note 
that each of these functions attains their single 
local maximum at $\hat{y}={\lambda_j x}(\alpha+k)^{-1}>0$, for all $x>0$. 
Therefore,
\begin{equation*}
\int_0^{\infty}I'_{jk}(x;y)\dd y-M_{jk}(x;\hat{y})\leq \sum_{i=1}^{\infty}x^{k}i^{-(\alpha+k)}\e^{-\lambda_j x/i}
\leq\int_0^{\infty}I'_{jk}(x;y)\dd y+M_{jk}(x;\hat{y}).
\end{equation*}
Observe that
\begin{equation*}
 M_{jk}(x;\hat{y})=x^k\e^{-(\alpha+k)}\left(\frac{\lambda_j}{\alpha+k}\right)^{-(\alpha+k)}x^{-(\alpha+k)}=cx^{-\alpha},
\end{equation*}
and
\begin{equation*}
 I'_{jk}(x):=x^k\int_0^{\infty}y^{-(\alpha+k)}\e^{-\lambda_j x/y}\dd y=\frac{\Gamma(\alpha+k-1)}{\lambda^{\alpha+k-1}}x^{-\alpha+1},
\end{equation*}
so $M_{jk}(x;\hat{y})$ is of negligible order with respect to $I'_{jk}(x)$.  Then it follows that
\begin{align*}
\Ftail(x)\sim \sum_{j=1}^{m}\sum_{k=0}^{\eta_j-1}\dfrac{c_{jk}}{\zeta(\alpha)}I'_{jk}(x)
=\sum_{j=1}^{m}\sum_{k=0}^{\eta_j-1}\dfrac{c_{jk}\Gamma(\alpha+k-1)}{\zeta(\alpha)\lambda^{\alpha+k-1}}x^{-\alpha+1},\quad x\to\infty.
\end{align*}
Thus $F(x)\in$MDA($\Phi_{\alpha-1}$).
Let $C=\sum\limits_{j=1}\limits^{m}\sum\limits_{k=0}\limits^{\eta_j-1}\dfrac{c_{jk}\Gamma(\alpha+k-1)}{\zeta(\alpha)\lambda^{\alpha+k-1}}$, then the norming constants can be chosen as 
\begin{equation*}
d_n=0,\quad c_n=\left(\frac{1}{\Ftail}\right)^{\leftarrow}(n)=\left(\frac{C}{n}\right)^{\frac{1}{\alpha-1}}.
\end{equation*}
\end{proof}

\begin{Example}[Geometric scaling]
Let $H\sim\mathrm{Geo}(p)$ and $G$ be PH distribution whose sub-intensity matrix
has only real eigenvalues. Then $F$ is a subexponential distribution 
in the Gumbel domain of attraction.
\end{Example}
\begin{proof}
Let
$p(i)=pq^i$ where $q=1-p$.
Since the geometric distribution has unbounded support, then the associated phase-type 
scale mixture is heavy-tailed. We next verify that it belongs to the Gumbel domain of attraction.
\begin{align*}
\Ftail(x)=\sum_{i=1}^{\infty}\Gtail\left(x/i\right)pq^i.
\end{align*} 
Let $I'(x;y)=\Gtail(x/y)p \exp\{-|\log q|y\}$  satisfies the conditions in Proposition \ref{integral approx}.  
Since the sine and cosine functions are bounded, then it is not difficult to use Proposition \ref{myprop2.1}
to show that there exists a constant $c_1$ such that
\begin{equation*}
M(x):=I(x;\hat{y})\le x^{\frac{k}{2}}\e^{-2\sqrt{x\lambda  |\log q|}}(c_1+\oh(1)),\qquad x\to\infty,
\end{equation*}
where $\lambda$ is the largest eigenvalue in absolute value and $k$ is its largest multiplicity.
If the sub-intensity matrix has real eigenvalues then by using Lemma 2.1 in \citep{ArendarczykDebicki2011} we obtain that
\begin{align*}
\int_0^{\infty}I'(x;y)\dd y&=p\int_0^\infty \Gtail(x/y)\e^{-y|\log q|}\dd y=
 x^{k/2+1/4}\e^{-2\sqrt{x\lambda|\log q|}}(C_1+\oh(1)),\qquad x\to\infty.
\end{align*} 
So, the value of $M(x)$ is asymptotically negligible with respect to the value of the integral
and we conclude that

\begin{equation*}
 \Ftail(x)\sim p\int_{0}^{\infty}\Gtail(x/y)\e^{-y|\log q|}\dd y =\frac{p}{|\log q|}\int_0^\infty \Gtail(x/y)\dd H(y),
\end{equation*}
where $H\sim \exp(|\log q|)$.  Hence, by tail equivalence, the distribution $F$ inherits all the asymptotic
properties of its continuous counterpart, namely, a phase-type scale distribution with exponential scaling
distribution with parameter $|\log q|$.
%Then it follows that
%\begin{align*}
% \Ftail(x)&\sim 
% 2pcx^{\eta-1}\left(\frac{|\log q|}{\lambda x}\right)^{\frac{\eta-2}{2}}
% \mathrm{BesselK}(\eta-2,2\sqrt{\lambda x|\log q|}),
%\end{align*}
%for some constant $c$.  Similar argument yields
%\begin{align*}
% F^{\prime}(x)&\sim 
% 2\lambda pcx^{\eta-1}\left(\frac{|\log q|}{\lambda x}\right)^{\frac{\eta-1}{2}}
% \mathrm{BesselK}(\eta-1,2\sqrt{\lambda x|\log q|}),\\
% F^{\prime\prime}&\sim 
%  -2\lambda^2pcx^{\eta-1}\left(\frac{|\log q|}{\lambda x}\right)^{\frac{\eta}{2}}
% \mathrm{BesselK}(\eta,2\sqrt{\lambda x|\log q|}).
%\end{align*}
%All modified Bessel functions involved above are asymptotically equivalent, so it
%follows that
%Then asymptotically, $\Ftail$, $F^{(1)}$ and $F^{(2)}$ can be expressed in terms of 
%equivalent modified Bessel functions. Thus
%\begin{equation*}
% \lim_{x\to\infty}\frac{\Ftail(x)F^{(2)}(x)}{\left(F^{(1)}(x)\right)^2}=-1,
%\end{equation*}
%so the distribution $F$ belongs to the Gumbel domain of attraction.
%
%Moreover, it follows that for any $t>1$,
%\begin{equation*}
%\lim_{x\to\infty}\dfrac{a(tx)}{a(x)}=\lim_{x\to\infty}\dfrac{\Ftail(tx)F^{(1)}(x)}{F^{(1)}(tx)\Ftail(x)}=\sqrt{t}>1.
%\end{equation*} 
%Thus F is subexponential distribution.
\end{proof}

\begin{Remark}
 We shall recall that the geometric version can be seen as the discrete counterpart of the exponential distribution
 obtained as a discretization.  More precisely, the geometric distribution can be seen as a 
 distribution supported over $\mathbb{Z}^+$ and defined by 
 \begin{equation*}
   H(k)=\mathcal{H}(k), \quad k=0,1,2,\cdots,
 \end{equation*}
 where $\mathcal{H}\sim\mbox{exp}(|\log q|)$. The probability mass function of $H$ is given by 
 $h(k)=\mathcal{H}(k)-\mathcal{H}(k-1)$.
 
 This idea can be extended in order to select scaling distributions for approximating heavy-tailed distributions in the Gumbel
 domain of attraction.  Suppose we want to approximate the tail probability of an absolutely continuous distribution $\mathcal{H}$
 supported over $(0,\infty)$ via a discrete phase-type scale mixture distribution. One way to proceed is to construct
 a discrete distribution supported over $\mathbb{N}$ defined by $h(k)=\mathcal{H}(k)-\mathcal{H}(k-1)$; we refer to this
 construction as a \emph{discretization} of $\mathcal{H}$.
 Moreover, the density of $\mathcal{H}$ can be used to construct a function $I'(u;k)$.   In such a case the tail behavior of a phase-type scale mixture
 having a discretized scaling distribution inherits the asymptotic properties of its continuous counterpart.

 This idea is better illustrated with the following example, which suggests a methodology for approximating the tail
 probability of a lognormal distribution.
 
\end{Remark}

\begin{Example}[Lognormal discretization]
Let $H$ be a discrete lognormal distribution with parameters $\mu$, $\sigma$ and supported over $\{0,1,2,\cdots\}$. Assume $\mu=0$.
The tail probability $\Ftail$ is given by
\begin{equation*}
\Ftail(x)=\sum_{i=1}^{\infty}\Gtail(x/i)\left[H(i)-H(i-1)\right]=\sum_{i=1}^{\infty}\Gtail(x/i)\int_{i-1}^ih(y)\dd y,
\end{equation*}
where $h(\cdot)$ is the density of lognormal distribution.
Since $\Gtail(x/y)$ is increasing in $y$, then we can easily find a lower bound: 
\begin{equation*}
\Ftail(x)= \sum_{i=1}^{\infty}\int_{i-1}^i\Gtail(x/i)h(y)\dd y\ge \int_0^{\infty}\Gtail(x/y)h(y)\dd y.
\end{equation*}

For the upper bound, we have
\begin{align*}
\Ftail(x)&\le \sum_{i=1}^{\infty}\int_{i-1}^i\Gtail(x/(y+1))h(y)\dd y=
\sum_{i=1}^{\infty}\int_{i-1}^i\Gtail(x/(y+1))[h(y)-h(y+1)+h(y+1)]\dd y\\
&\le\int_0^{\infty}\Gtail(x/y)h(y)\dd y+\int_0^{\infty}\Gtail(x/(y+1))[h(y)-h(y+1)]\dd y.
\end{align*}

For the second integral  in the above, we have
\begin{align*}
&\int_0^{\infty}\Gtail(x/(y+1))[h(y)-h(y+1)]\dd y\\
=&\int_0^{1}\Gtail(x/(y+1))[h(y)-h(y+1)]\dd y+\int_1^{\infty}\Gtail(x/(y+1))[h(y)-h(y+1)]\dd y\\
\le & c_1\Gtail(x/2)+c_2\int_1^{\infty}\Gtail(x/(y+1))\dfrac{h(y+1)}{(y+1)^{\beta}}\dd y,
\end{align*}
where $c_1$, $c_2>0$ are some constants and $0<\beta<1$.

It is not difficult to obtain this upper bound: firstly, it is easy to prove 
for $y\ge 1$, 
$\log(y+1)-\log(y)\le 1/y$, consequently, $\log^2(y+1)-\log^2(y)\le 2\log(y+1)/y$;
then we have 
\begin{align*}
\dfrac{h(y)}{h(y+1)}-1&=\dfrac{y+1}{y}\exp\left\{\dfrac{\log^2(y+1)-\log^2(y)}{2\sigma^2}\right\}-1\\
&\le \exp\left\{\dfrac{1}{y}+\dfrac{\log(y+1)}{\sigma^2y}\right\}-1\\
&\le c\left(\dfrac{1}{y}+\dfrac{\log(y+1)}{\sigma^2y}\right), \text{where}\ 
c>0\  \text{is some constant},\\
&\le \dfrac{c_2}{(y+1)^{\beta}}.
\end{align*}

Define 
\begin{equation*}
I'_{jk}(x):=x^k\int_0^{\infty}y^{-k}\e^{-\lambda_jx/y}h(y)\dd y.
\end{equation*}
From Example \ref{myex lognormal}, we know that 
\begin{align*}
\int_0^{\infty}\Gtail(x/y)h(y)\dd y &= \sum_{j=1}^{m}\sum_{k=0}^{\eta_j-1} c_{jk}\int_0^{\infty}(x/y)^k\e^{-\lambda_jx/y}h(y)\dd y=\sum_{j=1}^{m}\sum_{k=0}^{\eta_j-1} c_{jk}\dfrac{(-1)^kx^k}{\lambda_j^k}\mathcal{L}_Y^{(k)}(\lambda_jx)\\
&=\sum_{j=1}^{m}\sum_{k=0}^{\eta_j-1} c_{jk}\, \frac{x^k}{\lambda_j^k}\mathcal{L}_{Y}\exp\{-k\omega_0(\lambda_jx)+\dfrac{1}{2}\sigma_0(\lambda_jx)^2k^2\}.
%(1+\Oh(\sigma_0(\lambda_jx)^2)).
\end{align*}
So
\begin{align*}
I'_{jk}(x)
&=\left(\frac{x}{\lambda_j}\right)^k\mathcal{L}_Y\exp\{-k\omega_0(\lambda_jx)+\dfrac{1}{2}\sigma_0(\lambda_jx)^2k^2\}.
%(1+\Oh(\sigma_0(\lambda_jx)^2))\\
%&\sim \dfrac{(-1)^k}{(\lambda_j)^{2k}}\cdot \dfrac{\sigma^2}{1+\log(\lambda_jx%\sigma^2)},
\end{align*}
It is obvious that $c_1\Gtail(x/2)$ vanishes faster than $I'_{jk}(x)$,
so we can define
\begin{align*}
M_{jk}(x):=x^k\int_0^{\infty}y^{-k-\beta}\e^{-\lambda_jx/y}h(y)\dd y,
\end{align*}

%\begin{equation*}
%I'_{jk}(x;y)=\left(\dfrac{x}{y}\right)^k\e^{-\lambda_jx/y}\dfrac{1}{\sqrt{2\pi}\sigma y}\e^{-\frac{1}{2}\left(\frac{\log y}{\sigma}\right)^2},
%\end{equation*}
%it achieves the maximum value at $\widehat y=\e^{-(k+1)\sigma^2+\mathcal{W}(\lambda_jx\sigma^2\e^{(k+1)\sigma^2})}$ where $\mathcal{W}(\cdot)$ is the Lambert W function and 
%as $\cdot\to\infty$, $\mathcal{W}(\cdot)=\log(\cdot)(1+ o(1))$. 

%\correction{($\mathcal{W}(\cdot)=\log(\cdot)(1+ o(1))$, then \begin{equation*}
%\e^{\mathcal{W}(x)}=\e^{\log(x)(1+ o(1))}=\e^{\log(x)}\e^{\log(x)o(1)}=x\sum_{n=0}^{\infty}\dfrac{(o(\log(x))^n}{n!}=x(1+o(\log(x)))\end{equation*})}
%
%Then we can find 
%asymptotically the maximum of $I'_{jk}(x;y)$ is:
%\begin{equation*}
%M_{jk}(x)\sim \dfrac{\e^{-\sigma^{-2}}}{\sqrt{2\pi}\sigma}\left(\lambda_j\sigma^2\right)^{-(k+1)}\cdot\dfrac{1}{x}\ \e^{-\frac{1}{2}\left(\frac{\log\left(\lambda_jx\sigma^2\right)}{\sigma}\right)^2},
%\end{equation*} 
%and

since 
\begin{align*}
&c_2\int_1^{\infty}\Gtail(x/(y+1))\dfrac{h(y+1)}{(y+1)^{\beta}}\dd y
=c_2\int_2^{\infty}\Gtail(x/y)\dfrac{h(y)}{y^{\beta}}\dd y\\
&\le\sum_{j=1}^{m}\sum_{k=0}^{\eta_j-1} c_{jk}\int_0^{\infty}(x/y)^ky^{-\beta}\e^{-\lambda_jx/y}h(y)\dd y.
\end{align*}

By a similar approximation as in Example \ref{myex lognormal}, we can see
\begin{align*}
M_{jk}(x)&=(-1)^{k+\beta}\frac{x^k}{\lambda_j^{k+\beta}}\mathcal{L}_Y^{(k+\beta)}(\lambda_jx)\\
&=\frac{x^k}{\lambda_j^{k+\beta}}\mathcal{L}_Y\exp\{-(k+\beta)\omega_0(\lambda_jx)+\dfrac{1}{2}\sigma_0(\lambda_jx)^2(k+\beta)^2\}.
\end{align*}
%and
%\begin{equation*}
%M_{jk}(x)\sim \dfrac{(-1)^k}{(\lambda_j)^{2k}}\cdot \dfrac{\e^{-k}\sigma^2}{1+\log(\lambda_j\e x\sigma^2)}-\dfrac{(-1)^k}{(\lambda_j)^{2k}}\cdot \dfrac{\sigma^2}{1+\log(\lambda_jx\sigma^2)}
%\end{equation*}
So  $M_{jk}(x)$ is negligible compared to integral $I'_{jk}(x)$. 
Thus, the phase-type scale mixture distribution with discrete lognormal scaling 
has the same asymptotic behavior as the phase-type scale mixture distribution 
with lognormal scaling.
\end{Example}

\subsection{Non-lattice supports}

The examples in the previous subsection may suggest that a phase-type scale mixture having a discretized scaling distribution
will inherit the asymptotic properties of its continuous counterpart.  
However, such a discretization cannot be made arbitrarily.  The following example illustrates this fact.
\begin{Example}%(False regularly varying discretization).
Let $H\in\mathcal{R}_{-\alpha}$ be a continuous distribution and $S$ be a discrete random variable supported over $\{s_1,s_2,\cdots\}$ satisfying
\begin{equation*}
\Prob(S=s_i)=H(s_{i})-H(s_{i-1}), \quad i=1,2,\cdots.
\end{equation*}
%with $s_0=-\infty$.  
Suppose there exists $\epsilon>0$ and $i_0\in\nat$ such that $\forall i>i_0$,  it holds that
$s_{i+1}>s_i(1+\epsilon)$. Then
%Thus, $\Prob[S>x]=\Htail(s(x))$, where $s(x)=\min\{s_i:s_i\geq x\}$. 
%Let $s(x)=\min\{s_i:s_i\geq x\}$, then
\begin{align*}
\limsup_{x\to\infty}\dfrac{\Prob[S>(1+\epsilon)x]}{\Prob[S>x]}
 =\limsup_{i\to\infty}\dfrac{\Prob[S>(1+\epsilon)s_i]}{\Prob[S>s_i]}
%=&\lim_{i\to\infty}\dfrac{\Prob[S>s((1+\epsilon)s_i)]}{\Prob[S>s(s_i)]}
=\limsup_{i\to\infty}\dfrac{\Prob[S>s_{i}]}{\Prob[S>s_{i}]}
=1.%\neq(1+\epsilon)^{\alpha}.
\end{align*}
%The second last equation is true because $s_{i+1}>(1+\epsilon)s_i>s_i$. 
Then $S$ does not have a 
regularly varying distribution.
Suppose that $Y\sim\mbox{Erlang}(\lambda,k)$. 
According to Example 4.4 in \cite{MikoschJacobsenRosinski2009}, the distribution of phase-type scale mixture
random variable $S\cdot Y$ is not regularly varying.
\end{Example}

Nevertheless, such a discretization will provide a \emph{reasonable} approximation to a regularly varying distribution.
The following is a continuation of our previous example and it shows that such a distribution satisfies an analogue
of Breiman's lemma.

\begin{Example}\label{converse of Breiman}
Let $K>0$ and define $H_K$ a discrete distribution supported over $\{s_i:i\in\mathbb{Z}^+\}$,
where $s_i=\exp(i/K)$, and determined by
\begin{equation*}
 H_K(s_i)=1-s_{i}^{-\alpha}, \qquad \forall i\in\mathbb{Z}^+.
\end{equation*}
The distribution $H_K$ can be seen as a discretization 
over a geometric progression of a Pareto distribution having tail probability
$H(x)=x^{-\alpha}$ supported over $[1,\infty)$.  The following argument shows that $H_K$ is no longer a regularly varying distribution.  
Notice that for all $t>1$
there exist $n\in\mathbb{Z}^+$  such that $s_{n}< t\le s_{n+1}$, hence
\begin{align*}
\liminf_{x\to\infty}\frac{\Htail_K(xt)}{\Htail_K(x)}&
%=\liminf_{s_n\to\infty}\dfrac{\Htail(s_nt)}{\Htail(s_{n+1})}
%=\lim_{n\to\infty}\dfrac{\Htail( \e^nt)}{\Htail(\e^{n+1})}
%=\lim_{n\to\infty}\dfrac{\Htail(\e^{n+k+1})}{\Htail(\e^{n})}
=s^{-\alpha}_{n+1}, &
\limsup_{x\to\infty}\frac{\Htail_K(xt)}{\Htail_K(x)}&
%=\limsup_{s_n\to\infty}\dfrac{\Htail(s_nt)}{\Htail(s_n)}
%=\lim_{n\to\infty}\dfrac{\Htail( \e^nt)}{\Htail(\e^n)}
%=\lim_{n\to\infty}\dfrac{\Htail(\e^{n+k})}{\Htail(\e^n)}
=\begin{cases}s^{-\alpha}_n&t<s_{n+1}\\ s^{-\alpha}_{n+1}&t=s_{n+1}. \end{cases}
\end{align*}
Thus, according to Example 4.4 in \cite{MikoschJacobsenRosinski2009}, the Mellin--Stieltjes convolution of an Erlang 
distribution $G$ with the distribution $H$ given above is no longer of regular variation (the conditions described in 
Proposition \ref{integral approx} are not satisfied for this example either).  In spite of this, we can still analyse  
certain aspects of the asymptotic behavior of such a Mellin--Stieltjes convolution. For that purpose,
note that the following inequalities hold for all $w>1$
\begin{equation*}
 \e^{-\alpha/K}\Htail(w)< \Htail_K(w)\le \Htail(w),
\end{equation*}
hence we obtain that
\begin{align*}
 \e^{-\alpha/K}\int_0^\infty \Htail(x/s)dG(s)< \int_0^\infty \Htail_K(x/s)dG(s)\le \int_0^\infty \Htail(x/s)dG(s).
\end{align*}
Using Breiman's lemma we find that
\begin{align*}
 \e^{-\alpha/K}<\liminf \frac{\Ftail(x)}{M_G(\alpha)\Htail(x)}\le\limsup \frac{\Ftail(x)}{M_G(\alpha)\Htail(x)}\le 1.
\end{align*}
A heuristic interpretation of the inequalities above is that aysmptotically  the tail probability $\Ftail$ \emph{oscillates} 
between two regularly varying tails, so this example illustrates a behavior similar to that described
by Breiman's lemma.
Notice that the range of oscillation collapses as $K\to\infty$, which is consistent
with the fact that $H_K\to H$ weakly.  A better asymptotic approximation in the following argument is particularly sharp for numerical purposes.  Consider
\begin{align*}
\Ftail(x)=\int_0^{\infty}\Gtail\left(x/s\right)\dd H_K(s)&=(1-\e^{-\alpha/K})\sum_{i=0}^{\infty}\Gtail(x\e^{-i/K})\e^{-\alpha i/K}.
%\\
% &=(1-\e^{-\alpha})\sum_{n=0}^{\infty}\sum_{j=1}^{m}\sum_{k=0}^{\eta_j-1}c_{jk}(x\e^{-n})^k\e^{-\lambda_j x\e^{-n}}\e^{-\alpha n}\\
% &=(1-\e^{-\alpha})\sum_{j=1}^{m}\sum_{k=0}^{\eta_j-1}c_{jk}\sum_{n=0}^{\infty}(x\e^{-n/M})^k\e^{-\lambda_j x\e^{-n/M}}\e^{-\alpha n/M}.
\end{align*}
Let $I(x;i)=\Gtail(x\e^{-i/K})\e^{-\alpha i/K}$. The infinite series can be approximated via 
the integral 
\begin{align*}
\int_{0}^{\infty}I(x;y)\dd y
=\int_{0}^{\infty}\Gtail\left( x\e^{-y/K}\right)\e^{-\alpha y/K }\dd y
&=K\int_{1}^{\infty}\Gtail\left(\frac xw\right)w^{-(\alpha+1) }\dd w
 =\frac{K}\alpha\int_{1}^{\infty}\Gtail\left(\frac xw\right)\dd {H}(w).
\end{align*}
Since $G$ is such that $M_G(\alpha+\epsilon)<\infty$ for all $\epsilon>0$, then Breiman's lemma implies that
\begin{align*}
 {\Ftail(x)}&\approx \frac{1-\e^{-\alpha/K}}{\alpha/K}M_G(\alpha)\Htail(x).
\end{align*}
This approximation is consistent with the bounds found above, since for all $w>0$ it holds that
\begin{equation*}
 \e^{-w}\le \frac{1-\e^{-w}}{w}\le 1.
\end{equation*}
Hence, the asymptotic approximation suggested is contained in between the asymptotic bounds
previously found.

\end{Example}

The previous example demonstrates that the tail behavior of a phase-type scale mixture distribution
having a discretized scaling distribution is clearly affected by the selection of the support.
Naturally, better approximations will be obtained by taking a finer partition
of the support. 

The natural choice is to use a discretization of the target distribution over some lattice.
However, this approach is not always  suitable for numerical purposes,
because in practice there is only a finite number of terms of the infinite series that can be computed, so these
series are typically truncated.  
By selecting a discretization over a geometric progression,  we will obtain infinite series
that converge at faster rates, so these can be truncated earlier.  
More importantly, such geometric progressions still provide 
reasonable approximations of the tail probability as shown above.
This approach has been tested successfully in \cite{Nardo2016}, 
where they considered discretizing a Pareto distribution
over a geometric progression and used the corresponding phase-type scale mixture distribution to approximate Pareto claim size distributions in ruin probability calculations. This selection of the scaling
distribution is of critical importance in  \cite{BladtRojas2017} for estimating the parameters
of a phase-type scale mixture distribution via the EM algorithm. Such an estimation procedure
is iterative, so in each step it is necessary to compute a number of sufficient statistics involving 
these infinite series.  The selection of a geometric support allows us to compute the estimators within
a reasonable time.

\section{Conclusion}
\label{mysec5}
%SHORTER!
We considered the class of phase-type scale mixtures. Such distributions arise
from the product of two random variables $S\cdot Y$, where $S\sim H$ is a nonnegative
random variable and $Y\sim G$ is a phase-type random variable. Such a class is mathematically
tractable and can be used to approximate heavy-tailed distributions.

We provided a collection of results which can be used to determine the asymptotic behavior 
of a distribution in such a class. For instance, if the scaling distribution
$H$ has unbounded support, then the associated phase-type scale mixture distribution 
is heavy-tailed. We also provided verifiable conditions which can be employed
to classify the maximum domains of attraction and determine subexponentiality.
In particular, we were able to find phase-type scale mixture distributions with
equivalent asymptotic behavior for regularly varying and Weibullian distributions.  It is not
the case for the lognormal for which it is more difficult to suggest an appropriate scaling
distribution.

We considered the case of phase-type scale mixture distributions having discrete scaling distributions
since these are of critical importance in applications. We described a simple methodology which
allows to establish the asymptotic proportionality of these distributions with respect to their 
continuous counterparts. 
We exhibited important advantages and limitations of this approach
to approximate heavy-tailed distributions and analysed several important examples.

We remark that most of the results obtained here can be extended to an analogue class of 
\emph{matrix exponential scale mixture distributions} without too much effort. 
We note that some of our results were proven under the assumption that the phase-type distribution
has a sub-intensity matrix having only real eigenvalues.  Nevertheless, we conjecture that such results
holds for general phase-type and matrix-exponential distributions.
We also conjecture that a phase-type distribution is $\alpha$-regularly varying determining but 
this remains an open problem.

\section*{Acknowledgements}
 LRN is supported by Australian
Research Council (ARC) grant DE130100819.  WX is supported by IPRS/APA scholarship at The University of Queensland.

\bibliographystyle{chicagoa}
\bibliography{InfiniteBib,StandardBib}

\section{Appendix}
In this appendix we revise some classical results providing conditions for 
determining if a distribution belongs to the Gumbel domain of attraction and 
if it is subexponential.

A main result in extreme value theory
indicates that a distribution $H$ belongs to the Gumbel domain 
of attraction iff $\Htail$ is tail-equivalent to a von Mises function.
%\begin{equation}\label{cond.Gumbel}
%  \lim_{s\to\infty}\frac{\Htail(s)\Htail^{\prime\prime}(s)}{\left(\Htail^{\prime}(s)\right)^2}=-1, \quad \text{}
%\end{equation}
%given that $\Htail$ is twice differential and there exists $s_0$ such that $\Htail^{\prime\prime}(s)<0$ for all $s>s_0$.
%Moreover, von Mises functions are functions of rapid variation \citep[cf.][]{Bingham1987}.
The following provides sufficient conditions for a distribution to be a von Mises function.
 %\citep{EmbrechtsKluppelbergMikosch}.
\begin{theorem}[\cite{Haan1970}]\label{mydef2.2}
 Let $\Htail$ be a twice differentiable nonnegative distribution with unbounded support.  
 Then $\Htail$ is a von Mises function iff there exists $s_0$ such that
  $\Htail^{\prime\prime}(s)<0$ for all $s>s_0$, and
 \begin{equation}\label{cond.Gumbel}
  \lim_{s\to\infty}\frac{\Htail(s)\Htail^{\prime\prime}(s)}{\left(\Htail^{\prime}(s)\right)^2}=-1.
 \end{equation}
Moreover, von Mises functions are functions of rapid variation \citep[cf.][]{Bingham1987}.
\end{theorem}
%We are also interested in determining whether a distribution $H$ is subexponential.
\cite{GoldieResnick1988}
provide a sufficient condition for an absolutely continuous distribution $H\in\mathrm{MDA}(\Lambda)$ 
to be subexponential:
%\begin{equation}\label{cond.Gumbel.Sub}
%  \liminf_{s\to\infty}\frac{\Htail(ts)}{h(ts)}\frac{h(s)}{\Htail(s)}>1,\qquad \forall t>1,
%\end{equation}
%where $h$ is the density of $H$.
\begin{theorem}[\cite{GoldieResnick1988}]
\label{Th.GoldieResnick}
 Let $H\in\mathrm{MDA}(\Lambda)$ be an absolutely continuous function with density $h$, 
 then $H\in\mathcal{S}$ if
 \begin{equation}\label{cond.Gumbel.Sub}
  \liminf_{s\to\infty}\frac{\Htail(ts)}{h(ts)}\frac{h(s)}{\Htail(s)}>1,\qquad \forall t>1.
 \end{equation}
\end{theorem}

Therefore, since a phase-type scale mixture distribution is not only absolutely continuous but twice differentiable and
its second derivative is negative, 
then we can verify if it belongs to the Gumbel domain of attraction by just checking the condition 
\eqref{cond.Gumbel} 
in Theorem \ref{mydef2.2}.  
Subexponentiality can be checked via 
 condition \eqref{cond.Gumbel.Sub} in Theorem \ref{Th.GoldieResnick}. 
\end{document}